\documentclass [a4paper, 12pt]{amsart}

\usepackage[margin=2.3cm]{geometry} 
\usepackage[alphabetic]{amsrefs}
\usepackage{times}
\usepackage{wasysym,stmaryrd,enumerate,hyperref}

\usepackage{amsmath}
\usepackage{amsfonts}
\usepackage{amssymb}
\usepackage{amsthm}
\usepackage{graphicx}
\usepackage{pb-diagram}
\usepackage{epstopdf}
\usepackage{fancyhdr}
\usepackage[all]{xy}

\newtheorem{cor}{Corollary}[section]
\newtheorem{te}[cor]{Theorem}
\newtheorem{p}[cor]{Proposition}

\newtheorem{lemma}[cor]{Lemma}
\newtheorem*{mte}{Main Theorem}
\newtheorem{conjecture}[cor]{Conjecture}

\theoremstyle{definition}
\newtheorem{de}[cor]{Definition}
\newtheorem*{open}{Open problem}

\theoremstyle{remark}

\newtheorem{ex}[cor]{Example}

\newtheorem{nt}[cor]{Notation}

\newcommand{\cz}{\mathbb{C}}
\newcommand{\nz}{\mathbb{N}}
\newcommand{\zz}{\mathbb{Z}}

\newcommand{\ff}{\mathbb{F}}

\newcommand{\bb}{\mathcal{B}}

\newcommand{\pp}{\mathcal{P}}

\newcommand{\nr}{\mathcal{N}}

\newcommand{\vp}{\varepsilon}

   {\begin{verse}%
     \small }%
   {\end{verse}}

\DeclareMathOperator{\card}{Card} 
\DeclareMathOperator{\id}{id} 

\DefineSimpleKey{bib}{arx}
\BibSpec{arx}{%
  +{}{\PrintAuthors} {author}
  +{,}{ \textit} {title}
  +{}{ \parenthesize} {date}
  +{,}{ } {note}
  +{.}{ } {transition}
}

\begin{document}

\title{Almost commuting permutations are near commuting permutations}

\author{Goulnara Arzhantseva}
\address[G. Arzhantseva]{Universit\"at Wien, Fakult\"at f\"ur Mathematik\\
Oskar-Morgenstern-Platz 1, 1090 Wien, Austria.}
\email{goulnara.arzhantseva@univie.ac.at}
\author{Liviu P\u aunescu}
\address[L. P\u aunescu]{Institute of Mathematics of the Romanian Academy, 21 Calea Grivitei Street, 010702 Bucharest, Romania}
\email{liviu.paunescu@imar.ro}


\thanks{
G.A.\ was partially supported by the ERC grant ANALYTIC no.\ 259527. L.P. was partially supported by grant number PN-II-ID-PCE-2012-4-0201 of the Romanian National Authority for Scientific Research. Both authors were partially supported by the Austria-Romania research cooperation grant GALS on Sofic groups}

\begin{abstract}
We prove that the commutator is stable in permutations endowed with the Hamming distance, that is,  two permutations that almost commute are near two commuting permutations.
Our result extends to $k$-tuples of almost commuting permutations, for any given $k$, and allows restrictions, for instance, to even permutations. 
\end{abstract}

\maketitle
\section{Introduction}
A famous open problem asks whether or not two almost commuting matrices are necessarily close to two exactly commuting matrices.
This is considered independently of the matrix sizes and the terms ``almost'' and ``close'' are specified  with respect to a given norm. 
The problem naturally  generalizes to $k$-tuples of almost commuting matrices. It has also a quantifying aspect in estimating the required perturbation and
an algorithmic issue in searching for the commuting matrices whenever they do exist. 

The current literature on this problem, and its operator and $C^*$-algebras variants, is immense. 
The positive answers and counterexamples vary with matrices, matrix norms, and the underlying field,  we are interested in. 
For instance, for self-adjoint complex matrices and the operator norm the problem is due to Halmos~\cite{Halmos}. Its  affirmative solution for pairs of matrices is a major result of Lin~\cite{Lin},  see also~\cite{Lin_short}. A counterexample for triples of self-adjoint matrices was constructed by Davidson~\cite{Dav} and for pairs of unitary matrices, again with respect to 
the operator norm, by Voiculescu~\cite{Voi}, see also~\cite{EL}. For the normalized Hilbert-Schmidt norm on complex matrices, the question was explicitly formulated by Rosenthal~\cite{Ros}. Several 
affirmative and quantitative results for $k$-tuples of self-adjoint, unitary, and normal matrices with respect to this norm have been obtained recently~\cites{Had1, Had2,lev,  FK, FS}.

The problem is also renowned thanks to its connection to physics, originally noticed by von Neumann in his approach to quantum mechanics~\cite{vonN}. 
The commutator equation being an example, the existence of exactly commuting matrices near almost commuting matrices can be viewed in
a wider context of\emph{ stability} conceived by Ulam~\cite{Ulam}*{Chapter VI}:  an equation is stable if an almost solution (or a solution of the corresponding inequality)
is near an exact solution.

Our main result is the \emph{stability of the commutator}  in permutations endowed with the normalized Hamming distance, see Definition~\ref{def:H}  for details on that distance and Definition~\ref{def:stab} for a precise formulation of the notion of stability.

\begin{mte}
For any given $k\geqslant 2$ and with respect to the normalized Hamming distance,
every $k$ (even) permutations that almost commute are near $k$ commuting (respectively, even) permutations.
\end{mte}

The interest to the problem on the stability of the commutator in permutations has appeared very recently in the context of sofic groups~\cite{Gl-Ri}. 
Although, permutation matrices are unitary and the Hamming distance can be expressed using the Hilbert-Schmidt distance\footnote{We have $d_H(p,q)=\frac{1}{2}d_{HS}(A_p,A_q)^2$, where  $p,q\in Sym(n)$, and $A_p,A_q$ denote the corresponding $n\times n$ permutation matrices, $d_H$ is the Hamming distance, and $d_{HS}$ is the Hilbert-Schmidt distance, both normalized.}, the above mentioned
techniques available for unitary matrices, equipped with the Hilbert-Schmidt norm, do not provide successful tools towards the stability of the commutator in permutations. 

Our main theorem is the first stability result for the commutator in permutation matrices. 
A few related questions in permutations have been discussed in \cites{GS, M}.  However, this prior work takes a different viewpoint on commuting permutations
and does not yield any approach to the stability of the commutator. 

 Our proof of the main theorem relies on the ultraproduct technique, in particular, on the Loeb measure space construction introduced in \cites{L}.
 The arguments are valid for $k$-tuples of almost commuting permutations, for any given $k\geqslant 2$, and, under a slight adaptation, for
 $k$-tuples of almost commuting even permutations. 
 
 Stability results are useful to detect a certain rigidity of the corresponding classes of groups. 
 For instance, in~\cite{Gl-Ri} it was shown that a sofic stable (i.e. with a stable system of relator words) group  is residually finite.
We give our proof of this result using the ultraproduct language, see Theorem~\ref{stab_resfin}.   
We then introduce the concept of  \emph{weak stability,} see Definition~\ref{def:ws}. This notion suits better with
the study of metric approximations of groups (hence, in particular, that of sofic groups). It encompasses stability and allows us to
characterize weakly stable groups among amenable groups. This yields a new rigidity result.

\begin{te}[Theorem~\ref{thm:ws} (\ref{p})]\label{thm:p} 
Let $G=\ff_m/\langle R\rangle$ be an amenable group. Then $R$ is weakly stable if and only if $G$ is residually finite.
\end{te}

This result allows to provide many explicit examples of finite and infinite systems of relator words which are (not) weakly stable.
We collect some of these new examples in Section~\ref{sec:ex}. 

The relationship between stability and weak stability in permutations is intriguing.
We believe that a sound knowledge of all of the group quotients might be useful. 

\begin{conjecture}\label{conj:stable-weak stable}
A group $G$ is stable whenever every quotient $G/N$ is weakly stable.
\end{conjecture}

The commutator word being a specific example of a relator word, there are other relator words whose (weak)-stability will
be very interesting to determine. The following problem is rather challenging. Given group elements $u$ and $v$, we denote by $[u,v]=uvu^{-1}v^{-1}$
their commutator.

\begin{open}
Is the system of two words $[ab^{-1},a^{-1}ba] \hbox{ and }  [ab^{-1}, a^{-2}ba^2]$ (weakly) stable in permutations?
\end{open}

This system  is  the relator words of a finite presentation of the famous Thompson's group $F$ whose (non)-amenability question enchant
many mathematicians. An affirmative answer to the above problem, together with Theorem~\ref{thm:p} and Theorem~\ref{stab_resfin} (cf. \cite{Gl-Ri}*{Proposition 3})  respectively,
will imply that $F$ is not amenable and even not sofic. Whether or not a non sofic group does exist is a major open problem in the area
of metric approximations of infinite groups.

\section{Sofic groups and ultraproducts}
We begin with definitions from the theory of sofic groups and necessary reminders on the ultraproduct tools. 
We denote by $Sym(n)$ the symmetric group on a set with $n$ elements and by $Alt(n)$ its subgroup of even permutations.

\begin{de}\label{def:H}
For $p,q\in Sym(n)$ the normalized \emph{Hamming distance} is defined by:
\[d_H(p,q)=\frac1n\card\{i:p(i)\neq q(i)\}.\]
\end{de}

\begin{de}
A group $G$ is \emph{sofic} if $\forall$ finite subset $E\subseteq G$, $\forall\vp>0,$ there exists $n\in\nz^*$ and a map $\phi\colon E\to Sym(n)$ such that:
\begin{enumerate}
\item $\forall g,h\in E$ such that $gh\in E$, we have $d_H(\phi(g)\phi(h),\phi(gh))<\vp$;
\item $\forall g\in E$, such that $g\neq e$, we have $d_H(\phi(g),\id)>1-\vp$.
\end{enumerate}
\end{de}

Let $M_n=M_n(\cz)$ be the algebra of complex matrices in dimension $n$. For $a\in M_n$ define $Tr(a)=\frac1n\sum_ia(i,i)$. We identify the group $Sym(n)$ with $P_n\subseteq M_n$, the subgroup of permutation matrices. Observe that $d_H(p,\id)=1-Tr(p)$.

For $a\in M_n$ the \emph{trace norm},  aka the \emph{Frobenius norm}, is $\|a\|_2=\sqrt{Tr(a^*a)}=\sqrt{\frac1n\sum_{i,j}|a(i,j)|^2}$. We now construct the \emph{tracial ultraproduct} of matrix algebras. Let $\{n_k\}_k\subseteq\nz^*$ be a sequence of natural numbers such that $n_k\to\infty$ as $k\to\infty$, and $\Pi_kM_{n_k}$ be the Cartesian product. Let us consider the subset of bounded, in the operator norm $\|\cdot \|$, sequences of matrices: $l^\infty(\nz,M_{n_k})=\{(a_k)_k\in\Pi_kM_{n_k}: \sup_k\|a_k\|<\infty\}$. Let $\omega$ be a non-principal ultrafilter on $\nz$. We consider  $\nr_\omega=\{(a_k)_k\in l^\infty(\nz,M_{n_k}): \lim_{k\to\omega}\|a_k\|_2=0\}$, which is the ideal of $l^\infty(\nz,M_{n_k})$. The tracial ultraproduct of matrix algebras with respect to $\omega$ is defined as 
$\Pi_{k\to\omega}M_{n_k}=l^\infty(\nz,M_{n_k})/\nr_\omega$. The trace is then defined on the ultraproduct by $Tr(a) = \lim_{k\to\omega}Tr(a_k)$, where $a = \Pi_{k\to\omega} a_k\in \Pi_{k\to\omega}M_{n_k}$.

\subsection{The universal sofic group}
Various subsets of the tracial ultraproduct of matrix algebras will appear in this paper, first of which being the \emph{universal sofic group}, $\Pi_{k\to\omega}P_{n_k}\subseteq \Pi_{k\to\omega}M_{n_k}$, introduced by Elek and Szab\'o. 

\begin{te}[Theorem 1, \cite{El-Sz}] A group $G$ is sofic if and only if there exists an injective group homomorphism $\Theta\colon G\hookrightarrow\Pi_{k\to\omega}P_{n_k}$.
\end{te}

We call a group homomorphism $\Theta\colon G\to\Pi_{k\to\omega}P_{n_k}$ a \emph{sofic morphism} of $G$ and a group homomorphism $\Theta\colon G\hookrightarrow\Pi_{k\to\omega}P_{n_k}$ such that $Tr(\Theta(g))=0$ for all $g\neq e$ in $G$, a \emph{sofic representation} of $G$. Observe that a sofic representation is always injective.

\begin{de}
Two sofic representations $\Theta\colon G\hookrightarrow\Pi_{k\to\omega}P_{n_k}$ and $\Psi\colon G\hookrightarrow\Pi_{k\to\omega}P_{n_k}$  are said to be \emph{conjugate} if there exists $p\in \Pi_{k\to\omega}P_{n_k}$ such that 
$\Theta(g)= p\Psi(g)p^{-1}$ for every $g\in G$.
\end{de}

We shall use the following result of Elek and Szab\'o.
\begin{te}[Theorem 2, \cite{El-Sz2}]\label{thm:ez2} 
A finitely generated group $G$ is amenable if and only if any two sofic representations of $G$ are conjugate.
\end{te}

\subsection{The Loeb measure space}\label{Loeb space} 
We explain now the construction of the Loeb measure space. This space, introduced in \cites{L}, plays a  crucial role in our approach to the stability phenomenon. We present the Elek and Szegedy description, using ultralimits~\cites{El-Sze}.

Denote by $D_n\subseteq M_n$ the subalgebra of diagonal matrices. Let $X_{n_k}$ be a set with $n_k$ elements, and $\mu_{n_k}$ be the normalized cardinal measure such that $L^\infty(X_{n_k},\mu_{n_k})\simeq (D_{n_k},Tr)$. 

Let $X_\omega=\Pi_kX_{n_k}/\sim_\omega$ be the algebraic ultraproduct. That is, $\Pi_kX_{n_k}$ is the Cartesian product and $(x_k)_k\sim_\omega(y_k)_k$ if and only if $\{k:x_k=y_k\}\in\omega$. For $(x_k)_k\in\Pi_k X_{n_k}$, we denote by $(x_k)_\omega$ its $\sim_\omega$-equivalence class in $X_\omega$.

For a sequence of subsets $A_k\subseteq X_{n_k}$ consider the set $\{(x_k)_k\in\Pi_k X_{n_k}:\{k:x_k\in A_k\}\in\omega\}$. This set is closed under the equivalence relation $\sim_\omega$. 
Then we construct $(A_k)_\omega=\{(x_k)_\omega\in X_\omega:\{k:x_k\in A_k\}\in\omega\}$.  The collection of these sets $\bb_\omega^0=\{(A_k)_\omega:A_k\subset X_{n_k}\}$ is a Boolean algebra of subsets of $X_\omega$.
On $\bb_\omega^0$ we define the measure $\mu_\omega((A_k)_\omega)=\lim_{k\to\omega}\mu_{n_k}(A_k)$. Let $\bb_\omega$ be the completion of $\bb_\omega^0$ with respect to this measure. Then $(X_\omega,\bb_\omega,\mu_\omega)$ is the \emph{Loeb measure space}, a non separable probability space, which we denote briefly by $(X_\omega,\mu_\omega)$.

Note that $L^\infty(X_\omega,\mu_\omega)$ and $(\Pi_{k\to\omega}D_{n_k},Tr)$ are isomorphic as tracial von Neumann algebras. This observation provides an alternative construction of the Loeb measure space, starting from $(\Pi_{k\to\omega}D_{n_k},Tr)$ and using the fact that any abelian von Neumann algebra is isomorphic to $L^\infty(X,\mu)$ for some space with measure $(X,\mu)$.

\subsection{The universal sofic action}\label{universal sofic action}

The group $P_{n_k}$ is acting on $X_{n_k}$ by the definition of the symmetric group. We can construct the Cartesian product action $\Pi_kP_{n_k}\curvearrowright\Pi_kX_{n_k}$. Fix an element 
$\Pi_kp_k\in\Pi_kP_{n_k}$. Its action on $\Pi_kX_{n_k}$ is compatible with the equivalence relation $\sim_\omega$ and preserves the measure $\mu_\omega$. It follows that $\Pi_kp_k$ is acting on $(X_\omega,\mu_\omega)$. Moreover, if $\Pi_kp_k\in\nr_\omega$ then this action is trivial.

A sofic morphism $\Theta\colon G\to\Pi_{k\to\omega}P_{n_k}$ induces an action of the group $G$ on $(X_\omega,\mu_\omega)$, while a sofic representation induces a free action of $G$ on the same space. These actions are a crucial tool in our proof of the Main Theorem.

\section{Stability}

We denote by $\ff_m$ the free group of rank $m$ and by $\{x_1,\ldots, x_m\}$ its free generators. Elements in $\ff_m$ are denoted by $\xi$ and let $R=\{\xi_1,\ldots,\xi_k\}$ be a finite subset of $\ff_m$. Let $\langle R\rangle$ be the normal subgroup generated by $R$ inside $\ff_m$. We set $G=\ff_m/\langle R\rangle$. If $\xi\in\ff_m$, then $\hat\xi\in G$ is the image of $\xi$ under the canonical epimorphism $\ff_m\twoheadrightarrow G$.

\begin{nt}
If $H$ is a group and $p_1,\ldots, p_m\in H$, we denote by $\xi(p_1,\ldots,p_m)\in H$ the image of $\xi$ under the unique group homomorphism $\ff_m\rightarrow H$  such that 
$x_i\mapsto p_i$.
\end{nt}

We now define the notion of stability.

\begin{de}\label{def:stab}
Permutations $p_1,p_2,\ldots,p_m\in Sym(n)$ are a \emph{solution} of $R$ if:
$$\xi(p_1,\ldots,p_m)=\id_n,\ \forall \xi\in R,$$ where $\id_n$ denotes the identity element of $Sym(n).$

Permutations $p_1,p_2,\ldots,p_m\in Sym(n)$ are a $\delta$-\emph{solution} of $R$, for some $\delta>0$, if:
\[d_H(\xi(p_1,\ldots,p_m),\id_n)<\delta,\ \forall \xi\in R.\]

The system $R$ is called \emph{stable} (or \emph{stable in permutations}) if $\forall\ \vp>0\ \exists\ \delta>0\ \forall n\in\nz^*$  $\forall\ p_1,p_2,\ldots,p_m\in Sym(n)$ a $\delta$-solution of $R$, there exist $\tilde p_1,\ldots,\tilde p_m\in Sym(n)$ a solution of $R$ such that $d_H(p_i,\tilde p_i)<\vp$.

The group $G=\ff_m/\langle R\rangle$ is called \emph{stable} if  its set  of relator words $R$ is stable. \end{de}

The definition of stability does not depend on the particular choice of finite presentation of the group:
Tietze transformations preserve stability as the Hamming metric is bi-invariant.

\section{Perfect homomorphisms}

\begin{de}
A (not necessarily injective) group homomorphism  $\Theta\colon G\to\Pi_{k\to\omega}P_{n_k}$ is called \emph{perfect} if there exist $p_k^i\in P_{n_k}$, $i=1,\ldots, m$ such that $\{p_k^1,\ldots, p_k^m\}$ is a solution of $R$ for any $k\in\nz$ and $\Theta(\hat x_i)=\Pi_{k\to\omega}p_k^i$.
\end{de}

\begin{te} \label{car_te}
The set $R$ is stable if and only if any group homomorphism $\Theta\colon G\to\Pi_{k\to\omega}P_{n_k}$  is perfect.
\end{te}
\begin{proof}
Let $\Theta(\hat x_i)=\Pi_{k\to\omega}q_k^i$ for $i=1,\ldots, m$.  We have 
$\lim_{k\to\omega}\|\xi(q_k^1, \ldots, q_k^m)-\id\|_2=0$ for all $\xi\in R$ as $\Theta$ is a homomorphism.
This is equivalent to $\lim_{k\to\omega}d_H(\xi(q_k^1, \ldots, q_k^m), \id)=0$. 
We apply then the definition of stability to $\{q_k^1,\ldots,q_k^m\}$ to construct the required permutations.

For the reverse implication, assume that $R$ is not stable. Then there exists $\vp>0$ such that for any $\delta>0$ there exist $n\in\nz^*$ and $p_1,p_2,\ldots,p_m\in Sym(n)$ a $\delta$-solution of $R$ such that for each $\tilde p_1,\ldots,\tilde p_m\in Sym(n)$ a solution of $R$ we have $\sum_id_H(p_i,\tilde p_i)\geqslant\vp$. By choosing a sequence $\delta_k\to 0$, we can construct a group homomorphism $\Theta\colon G\to\Pi_{k\to\omega}P_{n_k}$ that does not satisfy the requirement of being perfect.
\end{proof}

Observe that this theorem remains valid, up to a clear reformulation of the concepts involved, for \emph{any} metric approximation of $G$.

The following result was proved in \cite{Gl-Ri}*{Proposition 3} using a different terminology.
Our definition of stability is equivalent to the one used in that paper (although, constants $\varepsilon$ and $\delta$ play reverse roles in these two definitions).

\begin{te}\label{stab_resfin}
Let $\Theta\colon G\to\Pi_{k\to\omega}P_{n_k}$ be a perfect injective homomorphism. Then $G$ is residually finite.
\end{te}
\begin{proof}
Let $p_k^i\in P_{n_k}$, $i=1,\ldots, m$, such that $\{p_k^1,\ldots, p_k^m\}$ is a solution of $R$ for any $k\in\nz$ and $\Theta(\hat x_i)=\Pi_{k\to\omega}p_k^i$. Then $\theta_k(\hat x_i)=p_k^i$ defines a homomorphism $\theta_k\colon G\to P_{n_k}$. Fix $g\in G$. If $\theta_k(g)=\id$ for all $k$, it follows that $\Theta(g)=\id$. This contradicts the injectivity of $\Theta$
whenever $g\not=e$.
\end{proof}

\section{Partial sofic representations}

For the proof of our main result we cut a sofic representation $\Theta\colon G\to\Pi_{k\to\omega}P_{n_k}$ by a commuting projection 
$\tilde a\in\Pi_{k\to\omega}D_{n_k}$, where $D_n\subseteq M_n$ is the subalgebra of diagonal matrices. This construction was used in \cite{Pa}, but we formalise these objects a little differently here.

For a $*$-algebra $A$, we denote by $\pp(A)$ the set of projections in $A$, $$\pp(A)=\{a\in A:a^2=a=a^*\}.$$ For instance, $\pp(D_n)$ is the set of diagonal matrices with only $0$ and $1$ entries.
The Cartesian product $\Pi_kM_{n_k}$ is an algebra with pointwise addition and multiplication. For an element $a\in l^\infty(\nz,M_{n_k})\subseteq \Pi_kM_{n_k}$, we denote by $\tilde a$ its image under the canonical projection onto $\Pi_{k\to\omega}M_{n_k}$.

\begin{de}
A \emph{partial permutation matrix} $p\in M_n$ is a matrix with $0$ and $1$ entries for which there exists $S\subseteq\{1,\ldots,n\}$ such that $p$ has exactly one non-zero entry (which is equal to $1$) on each row and column in $S$ and it is $0$ elsewhere. Alternatively, $p=qa$, where $q\in P_n$, $a\in\pp(D_n)$, and $qa=aq$. \end{de}

We denote by $PP_n^a\subseteq M_n$ the set of all partial permutation matrices associated to $a\in\pp(D_n)$.
The set $PP_n^a$ is a subgroup of $M_n$ isomorphic to $P_{nTr(a)}$.
%
We construct the group $\Pi_{k\to\omega}PP_{n_k}^{a_k}$ as a subgroup in $\Pi_{k\to\omega}M_{n_k}$. The identity in $\Pi_{k\to\omega}PP_{n_k}^{a_k}$ is $\tilde a$, where $a=\Pi_ka_k$.

\begin{de}
A group homomorphism $\Theta\colon G\to\Pi_{k\to\omega}PP_{n_k}^{a_k}$ is called a \emph{partial sofic morphism}. 
\end{de}

\begin{de}
A \emph{partial sofic representation} is a partial sofic morphism $\Theta\colon G\to\Pi_{k\to\omega}PP_{n_k}^{a_k}$ such that  $Tr(\Theta(g))=0$ for all $g\neq e$ in $G$.
\end{de}

\begin{de}
Let $\Psi\colon G\to\Pi_{k\to\omega}P_{n_k}$ be a group homomorphism, $\Psi=\Pi_{k\to\omega}\psi_k$. Let $a=\Pi_ka_k\in\pp(\Pi_k D_{n_k})$ be such that $\tilde a$ commutes with $\Psi$. Define  a partial sofic morphism $a\cdot\Psi$ as follows:\ $a\cdot\Psi\colon G\to\Pi_{k\to\omega}PP_{n_k}^{a_k}$, $a\cdot\Psi=\Pi_{k\to\omega}a_k\psi_k$.
\end{de}

It can be easily checked that any partial sofic morphism $\Theta\colon G\to\Pi_{k\to\omega}PP_{n_k}^{a_k}$ is obtained in this way: $\Theta= a\cdot\Psi$, where $\Psi\colon G\to\Pi_{k\to\omega}P_{n_k}$ is a group homomorphism and $\tilde a$ commutes with $\Psi$. Also, $\tilde a=\Theta(e)$.
In addition, any partial sofic representation $\Theta$ is a product $a\cdot\Psi$, where $\Psi\colon G\to\Pi_{k\to\omega}P_{n_k}$ is a sofic representation. We do not use these observations for the proof of the main result.

A partial sofic morphism/representation can be viewed as a usual sofic morphism/represen\-tation with some extra unused space, filled with $0$ entries. Thus, it makes sense to speak about a perfect partial sofic morphism. We cut partial sofic morphisms/representations by elements in $\Pi_kD_{n_k}$, instead of $\Pi_{k\to\omega}D_{n_k}$, as the property of being perfect may depend 
on the actual choice of $a=\Pi_ka_k\in\pp(\Pi_kD_{n_k})$, not only on its class $\tilde a$.

\begin{de}
A partial sofic morphism $a\cdot\Theta\colon G\to\Pi_{k\to\omega}PP_{n_k}^{a_k}$ is  called \emph{perfect} if there exists  $p_k^i\in PP_{n_k}^{a_k}$, $i=1,\ldots, m$ such that $\{p_k^1,\ldots, p_k^m\}$ is a solution of $R$ for every $k\in\nz$ and $a\cdot \Theta(\hat x_i)=\Pi_{k\to\omega}p_k^i$.
\end{de}

\begin{te}\label{partial perfect}
Let $\Theta\colon G\to\Pi_{k\to\omega}P_{n_k}$ be a group homomorphism. Let $\{a^j\}_j\subseteq\pp(\Pi_k D_{n_k})$ be a sequence of projections such that $\sum_ja^j=\id$ and $\tilde a^j\cdot\Theta=\Theta\cdot\tilde a^j$ for each $j$. If $a^j\cdot\Theta$ is a perfect partial sofic morphism for every $j$, then $\Theta$ is a perfect sofic morphism.
\end{te}
\begin{proof}
Let $a^j=\Pi_ka_k^j$. Then $\sum_ja^j=\id$ is equivalent to $\sum_ja_k^j=\id_{n_k}$ for each $k\in\nz$. Since $a^j\cdot\Theta$ is a perfect partial sofic morphism, there exist 
$p_k^{j,i}\in PP_{n_k}^{a_k^j}$, $i=1,\ldots, m$, such that
$\{p_k^{j,1},\ldots,p_k^{j,m}\}$ is a solution of $R$, such that $(p_k^{j,i})^*p_k^{j,i}=a_k^j$ and $a^j\cdot\Theta(\hat x_i)=\Pi_{k\to\omega}p_k^{j,i}$.
Define $p_k^i=\sum_jp_k^{j,i}$. Then $p_k^i\in P_{n_k}$ and $\{p_k^1,\ldots,p_k^m\}$ is a solution of $R$ for every $k\in\nz$. Moreover, 
$\Theta(\hat x_i)=\Pi_{k\to\omega}p_k^i$. It follows that $\Theta$ is perfect.
\end{proof}

\section{The commutator}

\begin{p}\label{rf}
Let $G=\ff_m/\langle R\rangle$ be a residually finite group. Then there exists $\Theta\colon G\to\Pi_{k\to\omega}P_{n_k}$ a perfect sofic representation for any given sequence $\{n_k\}_k\subseteq\nz^*$ such that $n_k\to\infty$ as $k\to\infty$. 
\end{p}
\begin{proof}
Choose a decreasing sequence of finite index subgroups of $G$ with trivial intersection. Denote by $\{m_j\}_j$ the sequence of finite indices. Let $\Psi\colon G\hookrightarrow\Pi_{j\to\omega}P_{m_j}$ be the associated sofic representation. It is clearly perfect. So, $\Psi(g)=\Pi_{j\to\omega}p_j^g$, with $\{p_j^{\hat x_1}, \ldots, p_j^{\hat x_m}\}$ a solution of $R$ for every $j$.

We construct a carefully chosen amplification of $\Psi$ that fits our sequence of dimensions. Using $n_k\to\infty$ as $k\to \infty$, we construct an increasing sequence $\{i_j\}_j$ such that $n_k>j\cdot m_j$ for any $k>i_j$. Next, for each $k$, $i_j<k\leqslant i_{j+1}$, let $n_k=c_km_j+r_k$, with $r_k<m_j$ and $c_k\geqslant j$. For each $g\in G$ and such $k$, construct $q_k^g=p_j^g\otimes \id_{c_k}\oplus \id_{r_k}\in P_{n_k}$. Then, if $p_j^g$ is not the identity (and, hence, with no fixed points as it  corresponds to the non-trivial left translation action on the finite quotient of $G$ of index $m_j$):
\[Tr(q_k^g)=\frac{r_k}{n_k}<\frac{m_j}{j\cdot m_j}=\frac1j.\]
It follows that $Tr(q_k^g)\to0$  as $k\to \infty$ for any $g\in G$. Define $\Theta\colon G\to\Pi_{k\to\omega}P_{n_k}$ by $\Theta(g)=\Pi_{k\to\omega}q_k^g$. Then $\Theta$ is a perfect sofic representation
for the sequence  $\{n_k\}_k\subseteq\nz^*$.
\end{proof}

\begin{cor}\label{cor:p}
Let $G=\ff_m/\langle R\rangle$ be an amenable, residually finite group. Then any sofic representation of $G$ is perfect.
\end{cor}
\begin{proof}
Let $\Theta\colon G\hookrightarrow\Pi_{k\to\omega}P_{n_k}$ be a sofic representation. Use the previous proposition to construct a perfect sofic representation $\Psi\colon G\hookrightarrow\Pi_{k\to\omega}P_{n_k}$. By the Elek-Szab\'o theorem~\cite{El-Sz2}*{Theorem 2}, two sofic representations of an amenable group are conjugate. Hence, $\Theta$ and $\Psi$ are conjugate. It is clear that a sofic representation conjugated to a perfect one is perfect.
\end{proof}

The next theorem is crucial. To prove it we use the following technical result.

\begin{lemma}
Let $\{c^j\}_j\subseteq\pp(\Pi_{k\to\omega}D_{n_k})$ be a sequence of projections such that $\sum_jc^j=\id$. Then there exist projections $a_k^j\in \pp(D_{n_k})$ such that $c^j=\Pi_{k\to\omega}a_k^j$ and $\sum_ja_k^j=\id_{n_k}$ for each $k\in\nz$.
\end{lemma}
\begin{proof}
The equation $\sum_jc^j=\id$ implies that $c^s=c^s+\sum_{j\neq s}c^sc^j$ for each $s\in\nz$. Since $\Pi_{k\to\omega}D_{n_k}$ is an abelian von Neumann algebra, it follows that $c^sc^j$ is a positive element. Then $c^sc^j$ must be equal to $0$ for $s\neq j$.

It is easy to check, using functional calculus, that every projection in an ultraproduct is an ultraproduct of projections. Then there exists $c_k^j\in\pp(D_{n_k})$ such that $c^j=\Pi_{k\to\omega}c_k^j$. For each $k\in\nz$,  we define inductively $a_k^j=c_k^j\cdot(\id_{n_k}-\sum_{s<j}a_k^s)$. Then $a_k^1=c_k^1$, so $\Pi_{k\to\omega}a_k^1=c^1$. By induction, $\Pi_{k\to\omega}a_k^j=\Pi_{k\to\omega}[c_k^j\cdot(\id_{n_k}-\sum_{s<j}a_k^s)]=c^j\cdot(\id-\sum_{s<j}c^s)=c^j$. 

Projections $\{a_k^j\}_j\subseteq \pp(D_{n_k})$ are diagonal matrices with only $0$ and $1$ entries. Moreover, by construction, $a_k^sa_k^j=0$ for  $s\neq j$. It follows that the $1$ entries in these matrices do not overlap. Then, for every $s$, $\sum_j a_k^j\geqslant\sum_{j<s}a_k^j$, so $\Pi_{k\to\omega}(\sum_j a_k^j)\geqslant \Pi_{k\to\omega}(\sum_{j<s}a_k^j)=\sum_{j<s}(\Pi_{k\to\omega}a_k^j)$. It follows that $\Pi_{k\to\omega}(\sum_j a_k^j)\geqslant\sum_j(\Pi_{k\to\omega}a_k^j)=\sum_j c^j=\id$, so $\Pi_{k\to\omega}(\sum_j a_k^j)=\id$.

Let $b_k=\id_{n_k}-\sum_ja_k^j$. Then $b_k\in\pp(D_{n_k})$ and $\Pi_{k\to\omega}b_k=\Pi_{k\to\omega}(\id_{n_k}-\sum_ja_k^j)=\id-\id=0$. We replace $a_k^1$ by $a_k^1+b_k$, for every $k$, to finish the proof.
\end{proof}

\begin{te}\label{thm:Z}
Let $\Theta\colon \zz^n\to\Pi_{k\to\omega}P_{n_k}$ be a group homomorphism. Then $\Theta$ is perfect.
\end{te}
\begin{proof}
Let $X_\omega$ be the Loeb measure space defined in Section \ref{Loeb space}, such that $L^\infty(X_\omega)\simeq\Pi_{k\to\omega}D_{n_k}$. Then $\Theta$ induces an action of $\zz^n$ on $X_\omega$, as explained in Section \ref{universal sofic action},
see also ~\cite{Pa}*{Sections 1.2 and 2.4}. For a subgroup $H\leqslant \zz^n$, we define:
\[A_H=\{x\in X_\omega: Stab(x)=H\}.\]
Let $c_H\in\Pi_{k\to\omega}D_{n_k}$ be the characteristic function of $A_H$. Then $\sum_Hc_H=\id$ and $\{A_H\}_{H\leqslant \zz^n}$ form a partition of $X_\omega$. By the previous lemma, we  find projections $a_k^H\in\pp(D_{n_k})$ such that: $\sum_Ha_k^H=\id_{n_k}$, for each $k$ and $c_H=\Pi_{k\to\omega}a_k^H$. Let $a^H=\Pi_ka_k^H\in\Pi_kD_{n_k}$, so that $\tilde a^H=c_H$.

As $H$ is a normal subgroup of $\zz^n$, it follows that $A_H$ is an invariant subset for the action induced by $\Theta$. That is, $\tilde a^H\Theta(g)=\Theta(g)\tilde a^H$ for every $g\in\zz^n$.  Thus, we can construct a partial sofic morphism $a^H\cdot\Theta\colon G\to\Pi_{k\to\omega}PP_{n_k}^{a_k^H}$.

For $h\in H$ we have $a^H\cdot\Theta(h)=a^H\cdot\Theta(e)=\tilde a^H$. Also, for $h\in G\setminus H$ we have $Tr(a^H\cdot\Theta(h))=0$. Then $a^H\cdot\Theta$ is a partial sofic representation of the quotient $\zz^2/H$. This group is an amenable, residually finite group. Hence, by Corollary~\ref{cor:p}, $a^H\cdot\Theta$ is perfect. As $\sum_Ha^H=\id$, it follows by Theorem \ref{partial perfect} that $\Theta$ is perfect.
\end{proof}

The following 2 consequences give our Main Theorem from the Introduction.

\begin{cor}
 The commutator is stable in permutations endowed with the Hamming distance.
 Moreover, for any given $k\geqslant 2$, every $k$ permutations that almost commute are near $k$ commuting permutations.
 \end{cor}
\begin{proof}
This follows from the preceding theorem together with Theorem \ref{car_te}.
\end{proof}

\begin{cor}
 The commutator is stable in even permutations endowed with the Hamming distance.
 Moreover, for any given $k\geqslant 2$, every $k$ even permutations that almost commute are near $k$ even commuting permutations.
 \end{cor}
\begin{proof}
We proceed as above with a slight adaptation of details in the arguments. In the proof of Proposition~\ref{rf}, we
choose even numbers $c_k$. This ensures that the constructed permutations $q_g^k$ are even. 
Now Theorem~\ref{thm:Z} can be stated for
even permutations. This corollary then follows from this variant of Theorem~\ref{thm:Z} together with 
Theorem \ref{car_te}  restricted to even permutations.

We observe that having even permutations instead of arbitrary permutations as a $\delta$-solution is
not an extra hypothesis. Indeed, our proof above actually shows that every $k$ commuting permutations 
are near $k$  commuting even permutations.
\end{proof}

A careful analysis of the proof of Theorem~\ref{thm:Z} shows
that we have used the following (strong) properties of
$\mathbb{Z}^n$: (i) every subgroup is normal, (ii) every quotient is amenable and residually finite.
These properties are clearly satisfied by every quotient of $\mathbb{Z}^n$.

\begin{cor}
Every finitely generated abelian group is stable.
\end{cor}

\section{Weak stability}

We notice that in Theorem \ref{stab_resfin}, the existence of a perfect sofic representation is enough to deduce that the group is residually finite.  Combined with our proof of Theorem \ref{car_te}, this suggests the following a priori weaker version of stability. Let $l\colon \ff_m\to\nz$ be the word length function.

\begin{de}\label{def:ws}

Permutations $p_1,p_2,\ldots,p_m\in Sym(n)$ are a $\delta$-\emph{strong solution} of $R$, for some $\delta>0$, if for every $\xi\in\ff_m$ such that $l(\xi)<1/\delta$ we have:
\begin{align*}
\xi\in \langle R\rangle&\Longrightarrow d_H(\xi(p_1,\ldots,p_m),\id_n)<\delta;\\
\xi\notin \langle R\rangle&\Longrightarrow d_H(\xi(p_1,\ldots,p_m),\id_n)>1-\delta.
\end{align*}
The system $R$ is called \emph{weakly stable}  (or \emph{weakly stable in permutations}) if $\forall\ \vp>0\ \exists\ \delta>0\ \forall n\in\nz^*$  $\forall\ p_1,p_2,\ldots,p_m\in Sym(n)$ a $\delta$-strong solution of $R$ there exist $\tilde p_1,\ldots,\tilde p_m\in Sym(n)$ a solution of $R$ such that $d_H(p_i,\tilde p_i)<\vp$.

The group $G=\ff_m/\langle R\rangle$ is called \emph{weakly stable} if  its set  of relator words $R$ is stable.
\end{de}

Similarly to stability, the definition of weak stability does not depend on the particular choice of finite presentation of the group as Tietze transformations preserve weak stability. 

Comparing the definitions, it is clear that ``$R$ is stable" implies ``$R$ is weakly stable". It is hard, at this point, to say whether or not the converse is true, see also Conjecture \ref{conj:stable-weak stable}.

\begin{te}\label{thm:ws} Let $G=\ff_m/\langle R\rangle$ be a group.
\begin{enumerate}[(i)]
\item $G$ is weakly stable if and only if every sofic representation of $G$ is perfect.\label{car_te2}
\item If $G$ is sofic and $R$ is weakly stable, then $G$ is residually finite.\label{car_te3}
\item Suppose $G$ is amenable. Then $R$ is weakly stable if and only if $G$ is residually finite.\label{p}
\end{enumerate}
\end{te}

\begin{proof}
Proceeding as in the proofs of Theorems \ref{car_te} and \ref{stab_resfin}, we obtain (\ref{car_te2}) and (\ref{car_te3}). Let us check (\ref{p}). If $R$ is weakly stable then use (\ref{car_te3}). If $G$ is residually finite then use Corollary \ref{cor:p} and (\ref{car_te2}).
\end{proof}

\subsection{Examples of non (weakly) stable systems}\label{sec:ex}
 Our main theorem on the stability of the commutator is the first stability result in permutation matrices.
We list now other new examples of systems of relator words
which are  (not) (weakly) stable in permutation matrices as follows from our rigidity results, Theorem~\ref{thm:p} and Theorem~\ref{stab_resfin}.

\begin{ex}[Baumslag-Solitar groups]
These are groups defined by presentations $BS(m,n)=\langle a,t \mid t^{-1}a^mta^{-n}=1\rangle,$ where $m$ and $n$ are integers.
They are all sofic as they are known to be residually amenable. 
Also, $BS(m,n)$ is residually finite if and only if $\vert m\vert=\vert n\vert$ or $\vert m\vert=1$ or $\vert n\vert=1.$
Let $r(m,n)=t^{-1}a^mta^{-n}$ denotes the relator word. Using our results, we conclude that
\begin{itemize}
\item $r(m,n)$ is weakly stable in permutations whenever $\vert m\vert=1$ or $\vert n\vert=1.$
\item $r(m,n)$ is stable in permutations whenever $m= n=\pm 1.$
\item $r(m,n)$ is not stable in permutations whenever $\vert n\vert, \vert m\vert\geqslant 2$ and $\vert m\vert\not =\vert n\vert$.
\end{itemize}
\end{ex}

It is not yet known whether the fundamental group of the Klein bottle given by the presentation  $\langle a,t \mid t^{-1}ata=1\rangle\simeq BS(1,-1)\simeq BS(-1,1)$
is stable. For $\vert n\vert, \vert m\vert\geqslant 2$ and $\vert m\vert=\vert n\vert$, $BS(m,n)$ contains
a finite index subgroup isomorphic to $F_{\vert n\vert}\times \mathbb{Z}$,
the direct product of the free group of rank $\vert n\vert$ and the group of integers. The (weak) stability of the corresponding system of relator words
is an open question.


\begin{ex}[Sofic amalgamated products]
For $n\geqslant 3$ and $p$ prime, the amalgamated product
$
SL_n(\mathbb Z[{\textstyle\frac{1}{p}}])\ast_{\mathbb Z} SL_n(\mathbb Z[{\textstyle\frac{1}{p}}])
$
is finitely presented, sofic, and not residually amenable~\cite{Kar}. 
By Theorem~\ref{stab_resfin}, the finite system of relator words of this group is not stable in permutations.
\end{ex}


\begin{ex}[Kharlampovich's group]
This is a famous solvable group of class 3 which is finitely presented and has unsolvable word problem~\cite{Ha}.
In particular, it is an amenable finitely presented group which is not residually finite.
By Theorem~\ref{thm:p}, the finite system of relator words of this group is not weakly stable in permutations.
\end{ex}

\end{document}